\documentclass[11pt,a4paper,oneside,onecolumn]{article}
\usepackage{graphicx,amssymb, amsmath,theorem, calc}
\usepackage{colortbl}
\usepackage{multicol}
\usepackage{caption}
\usepackage{lineno}
\usepackage{subcaption}
\usepackage{yhmath}
\usepackage{hyperref}
\usepackage{xspace}
\usepackage{color}

\usepackage[letterpaper]{geometry}
\newtheorem{theorem}{Theorem}[section]

\newtheorem{corollary}[theorem]{Corollary}

\newtheorem{lemma}[theorem]{Lemma}

\newtheorem{problem}[theorem]{Problem}

\newtheorem{remark}[theorem]{Remark}

\def\QED{\ensuremath{{\square}}}
\def\markatright#1{\leavevmode\unskip\nobreak\quad\hspace*{\fill}{#1}}

\newenvironment{proof}
{\begin{trivlist}\item[\hskip\labelsep{\bf Proof.}]}
{\markatright{\QED}\end{trivlist}}

\newtheorem{theorem*}{Theorem}
\newcommand{\etal}{\emph{et al.}\xspace}

\newcommand{\arc}[1]{\stackrel{\text{\resizebox{\widthof{$#1$}}{2pt}{$\frown$}}}{#1}}
\newcommand{\seq}[2]{\underbrace{#1,\enspace{\scalebox{1.5}{$\cdots$}}\enspace, #1}_{#2}}

\begin{document}


\title{Asymmetric polygons with maximum area}
 

\author{L. Barba\thanks{Universit\'e Libre de Bruxelles (ULB), Brussels, Belgium. {\tt lbarbafl@ulb.ac.be}} \and
 L.E. Caraballo\thanks{Departamento de Matem\'{a}tica
Aplicada II, Universidad de Sevilla, Spain. {\tt luisevaristocaraballo-ext@us.es}} \and
J. M. D\'{\i}az-B\'{a}\~{n}ez\thanks{Departamento de Matem\'{a}tica
Aplicada II, Universidad de Sevilla, Spain. Partially supported by
projects FEDER P09-TIC-4840 (Junta de Andaluc\'{\i}a). {\tt dbanez@us.es.}} \and
R.~Fabila-Monroy\thanks{Departamento de
Matem\'aticas. Centro de Investigaci\'on y de Estudios Avanzados del Instituto Polit\'ecnico Nacional, Mexico City, Mexico. Partially supported by Conacyt of Mexico, grant 153984.
 {\tt ruyfabila@math.cinvestav.edu.mx}}\and
E. P\'erez-Castillo\thanks{Departamento de Matem\'{a}tica
Aplicada II, Universidad de Sevilla, Spain. {\tt  edepercas@alum.us.es}}
}

\date{\today}

\maketitle

\begin{abstract}
We say that a polygon inscribed in the circle is asymmetric if it contains no two antipodal points being the endpoints of a diameter. Given $n$ diameters of a circle and a positive integer $k<n$, this paper addresses the problem of computing a maximum area asymmetric $k$-gon having as vertices $k<n$ endpoints of the given diameters. The study of this type of polygons is motivated by ethnomusiciological applications.
\end{abstract}


\section{Introduction}

Imagine that we are at a concert of salsa and we want to retain the intrinsic rhythmic pattern to dance appropriately. Then we should know the clave Son.
The clave Son rhythm might be represented as the 16-bit binary sequence $1001001000101000$ or,
as usual for cyclic rhythms, by onsets represented as points on a circle as in Figure \ref{fig:buleson} c). In this geometric setting, the clave Son is associated to a pentagon whose vertices are selected on a circular lattice of 16 points. So, the clave Son is represented by a special selection of $k$ points on the circular lattice. 
This pattern has conquered our planet during the last half of the 20th century. But, what properties does this particular selection have?
Notice that the polygon that represents the clave Son does not contain two antipodal points on the circle and moreover, it is easy to prove that this configuration is just the pentagon of maximum area without antipodal vertices (this later property produces a certain kind of asymmetry). 
Musicians have showed interest in finding similar patterns. Ethnomusicology is the discipline encompassing various approaches to the study of music that emphasize its cultural context. More specifically, 
Ethnomathematics is a domain consisting of the study of mathematical ideas shared by orally transmitted cultures. Such ideas are related to numeric, logic and spatial configurations \cite{ascher1998,chemillier2002}.
%
Related to spacial configurations, the area is a useful measure of evenness of scales and rhythms in music theory \cite{rappaport, arom}. 
Moreover, many musical traditions all over the world have asymmetric
rhythmic patterns. For instance, the Aka Pygmies 
rhythmic pattern in Figure \ref{fig:buleson} b) has the so-called rhythmic oddity property discovered by Simha Arom \cite{arom}. A rhythm has the \emph{rhythmic oddity} property if when represented on a circle it does not contain two onsets (the black points in the Figure \ref{fig:buleson}) that lie diametrically opposite to each other. Thus, the property asserts that one cannot break
the circle into two parts of equal length, whatever the chosen
breaking point, as there is always one unit lacking on one side. This property produces a kind of perceptual asymmetry. The asymmetry of the pattern is to some extent
intrinsic, in the sense that there exists no breaking point
giving two parts of equal length. Note that the oddity
property requires that the circle is divided into an even
number of units. 
The notion of rhythmic oddity has received different
mathematical treatments; see \cite{toussaint} for more details.

An algorithm
for enumerating all the patterns satisfying the rhythmic
oddity property has been proposed in \cite{chemillier}.
Asymmetric rhythmics can also be found in the flamenco music of Spain (Figure \ref{fig:buleson} a)) and the \emph{clave Son}  in Cuba (Figure \ref{fig:buleson} c)). See \cite{diaz2004compas,diaz2013}, and \cite{toussaint2005}, for a detailed study on the preference of theses rhythms in their cultural contexts.

 \begin{figure}{th}
\begin{center}
 \includegraphics[width=.8\textwidth]{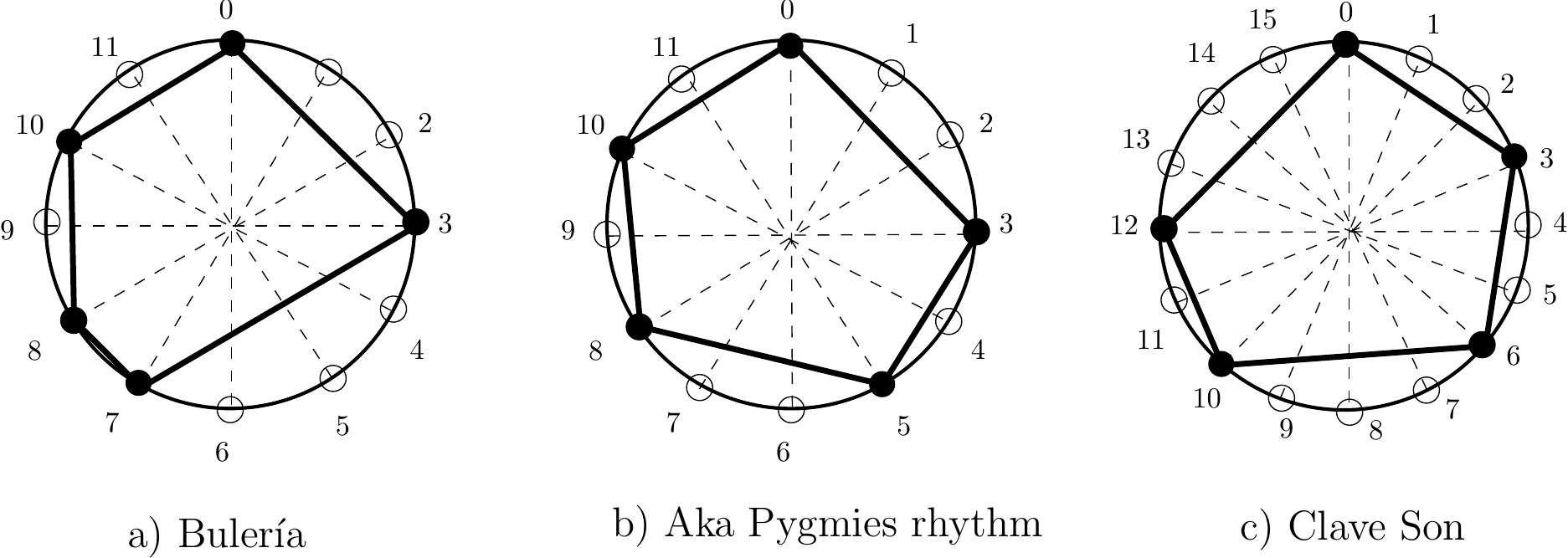}
 \caption{a) A flamenco
rhythm in Spain, b) a rhythm used in Central Africa, c) the clave son in Cuba.}\label{fig:buleson}
\end{center}
\end{figure}

Inspired in these ideas, we introduce the following geometric problem: 

\begin{problem}\label{problem: Main problem}
Given $n$ diameters in a circle and a positive integer $k<n$, select $k$ endpoints of these diameters, no two of the same diameter, in such a way that their convex hull defines a $k$-gon of maximum area.
\end{problem}

Let us introduce some notation and related work. Let $(p_0, p'_0), \ldots, (p_{n-1}, p_{n-1}')$ be $n$ diameters of a given circle.
Let $S:=\{p_0,p_0',p_1,p_1',\dots,p_{n-1},p_{n-1}'\}$ be an \emph{antipodal set} containing the $2n$ endpoints of the given diameters.
A \emph{sub-polygon} (of $S$) is a convex polygon whose vertex set is a subset of $S$.
An \emph{antipodal polygon} (on $S$)
is a sub-polygon whose vertex set contains exactly one endpoint
from each diameter $(p_i,p_i')$ of $S$~\cite{aichholzer2013}. 
An \emph{asymmetric polygon} is a sub-polygon that contains no diameter. 
Therefore, an antipodal polygon is also an asymmetric polygon.
Aichholzer \etal proved that an antipodal polygon of maximum area can be found in $\Theta(n)$-time~\cite{aichholzer2013}. 
The linear time algorithms they proposed are strongly based on a simple characterization for the extremal antipodal polygons. Namely, that an antipodal polygon of maximum area has an alternating configuration~\cite{aichholzer2013}. 

The problem however, is significantly different if we ask for an asymmetric $k$-gon of maximum area with $k<n$ vertices in $S$. It is not difficult to come up with examples for which the simple characterization stated above does not hold if $k<n$. 
Aichholzer \etal presented an $O(n^{n-k+1})$-time algorithm to compute an antipodal $k$-gon of maximum area. However, the existence of a polynomial time algorithm to solve this problem was left as an open question~\cite{aichholzer2013}.

In this paper, we answer this question in the affirmative. We distinguish two cases: If we are given a circular lattice with an antipodal set of $2n$ points (induced by $n$ evenly spaced diameters), we show how to solve Problem~\ref{problem: Main problem} in constant time by providing a characterization of the solution. Otherwise, if the diameters are given in a general configuration, we show that the problem can be solved in $O(kn^4)$-time using dynamic programming.

The problem studied here is related to other optimization problems in mathematics. In computational geometry, efficient algorithms have been proposed for computing extremal polygons with respect to several different properties~\cite{boyce}. Moreover,  the so-called stabbing or transversal problems (see for instance \cite{arkin}) belong to the same family that our problem. Recently, it has been proved that the following problem is NP-hard \cite{pilz}: 
 Given a set $S$ of line segments, compute the minimum or maximum area (perimeter) polygon $P$ such that $P$ stabs $S$, that is, at least one of the two endpoints of every segment $s\in S$ is contained in $P$.
 In operations research, global optimization techniques have been extensively studied
 to find convex polygons maximizing a given parameter~\cite{hansen}. Now, let us consider a different interpretation of the problem, let $P$ be the convex hull of the given diameters. In this case, the solution to Problem~\ref{problem: Main problem} can be
interpreted as the asymmetric polygon with $k<n$ vertices that has its area closest to that of $P$. 
Therefore, this problem is related to the approximation of convex sets.  In this setting,  the best inner approximation of any convex set by a symmetric set is studied in
\cite{lassak2002}. Moreover,
if we consider the ``distance'' to a symmetric set to be a measure of its symmetry \cite{grunbaum}, then our solution to Problem~\ref{problem: Main problem} provides the best approximation of a convex polygon inscribed in a circle by asymmetric sub-polygons.

Finally, it is worth mentioning that although there exists a high number
of applications of mathematics to music theory, the
research in music has illuminated problems that are appealing,
nontrivial, and, in some cases, connected to deep mathematical
questions. The problem introduced in this paper could be an example.

The remainder of the paper is organized as follows. In Section~\ref{sec:Evenly spaced} we consider the constrained version of Problem~\ref{problem: Main problem} in which the endpoints of the diameters are evenly spaced on the circle. In this case, we provide a characterization of the maximum $k$-gon which yields a constant time algorithm to solve Problem~\ref{problem: Main problem}.
The general version of this problem where the endpoints of the diameters are distributed anywhere on the circle is studied in Section~\ref{sec:General Case}. In this case, we show that Problem~\ref{problem: Main problem} can be solved in polynomial time using dynamic programming.

\section{Evenly spaced diameters}\label{sec:Evenly spaced}

In this section, we assume that $S=\{p_0,p_0',p_1,p_1',\dots,p_{n-1},p_{n-1}'\}$ is the set of endpoints of $n$ evenly spaced diameters on a unit circle. That is, $S$ partition the circle into $2n$ arcs of equal length. A $k$-gon $Q_k$ with vertices $q_{0},\ldots, q_{k-1}$ in $S$ can be encoded by a sequence $<a_{0}, ..., a_{k-1}>$, where $a_i$ is the number of arcs between the vertices $q_i$ and $q_{i+1}$, $0\leq i < k$, $q_{k}=q_{0}$. For example, the clave Son in Figure \ref{fig:buleson} c) can be encoded by the sequence $<3,3,4,2,4>$. In music theory, this sequence is called the interval vector (or full-interval vector) \cite{mccartin}. As with rhythmic patters, we assume that the sequence $<a_{0},\ldots, a_{k-1}>$ is cyclic, i.e., for $i\geq k$, $a_i=a_{i\bmod k}$. Throughout this paper, we identify a sub-polygon with $k$ vertices by its corresponding encoding sequence.

\begin{remark} Let $Q_k$ be a sub-polygon encoded by the sequence $<a_0,\dots, a_{k-1}>$. The following properties hold:
\begin{enumerate}
\item The area of the polygon $Q_k$ is $\dfrac{1}{2}\sum^{k-1}_{i=0}\sin(\frac{\pi }{n} a_{i})$.
\item The area is invariant under permutations of the sequence $<a_{0},\dots, a_{k-1}>$. That is, for any $0\leq i < j \leq k-1$, the area of the polygons given by the sequences $<a_{0},\ldots, a_i,\ldots, a_j,\ldots, a_{k-1}>$ and $<a_{0},\ldots, a_j,\ldots, a_i,\ldots, a_{k-1}>$ is the same. 
\item Polygon $Q_k$ is asymmetric if and only if there are no two indices $i$, $j$ such that $i\not=j$ and $a_{i}+a_{i+1}+\dots+a_{j}=n$.
\end{enumerate}
\end{remark}

The following technical results will be used later.

\begin{lemma}\label{lemma_swap}
Let $a$, $b$, $n$ be positive integers with $a-b\geq 2$ and $2n>a+b$. Then we have $$\sin\left(\frac{\pi a}{n}\right)+\sin\left(\frac{\pi b}{n}\right) < \sin\left (\frac{\pi (a-1)}{n}\right)+\sin\left(\frac{\pi (b+1)}{n}\right)$$
\end{lemma}
\begin{proof}
The previous inequality can be written as:
\begin{displaymath}
2\sin\left(\frac{\pi (a+b)}{2n}\right)\cos\left(\frac{\pi (a-b)}{2n}\right)< 2\sin\left(\frac{\pi (a+b)}{2n}\right)\cos\left(\frac{\pi (a-b-2)}{2n}\right)\ .
\end{displaymath}
Because $0 < \frac{\pi (a+b)}{2n}< \pi$, $0\leq \frac{\pi (a-b-2)}{2n} < \frac{\pi (a-b)}{2n} < \pi $ and from the fact that $\cos(\cdot)$ is decreasing in $(0,\pi)$, this inequality holds.
\end{proof}

For ease of notation, a sequence $a,\dots,a$ of $c$ elements all equal to $a$ is denoted by \;\raisebox{8pt}{$\underbrace{a,\dots,a}_{c}$}. The following lemma characterizes the $k$-gons (not necessarily asymmetric) of maximum area.

\begin{lemma}\label{lemma_max}
Given a regular polygon of $m$ vertices inscribed in a circle, the maximum area sub-polygon with $k\leq m$ vertices is encoded by any permutation of the following sequence 
$$<\seq{q+1}{r},\:\seq{q}{k-r}>\ ,$$
where $m=kq+r$, $0\leq r < k$.
\end{lemma}
\begin{proof}
Let $a_0,a_1,\ldots,a_{k-1}$ be a sequence encoding a sub-polygon $Q_k$ with maximum area. 
Define $a_{max} = \max\left(a_0,a_1,..., a_{k-1}\right)$ and $a_{min} = \min\left(a_0,a_1,...,a_{k-1}\right)$. 
If $a_{max} - a_{min} \geq 2$, then we switch $a_{max}$ to $a_{max}-1$, and $a_{min}$ to $a_{min}+1$. In this way, because $k\geq 3$ we have $m=\sum^{k-1}_{i=0}a_i>a_{max}+a_{min}$ and hence, Lemma \ref{lemma_swap} implies that the new polygon has larger area than the initial one leading to a contradiction. Therefore, $a_{max} - a_{min} \leq 1$. This implies that for some integer $e$ and for every $0\leq i\leq k-1$, either $a_i=e$ or $a_i=e+1$. 

Let $u$ and $v$ be the number of occurrences of the integers $e$ and $e+1$ respectively, that is, $u(e)+v(e+1)=m$. Thus, we have $(u+v)e+v=kq+r$ and $ke+v=kq+r$. Consequently, $v-r$ is a multiple of $k$. Since $v\leq k$ and $r<k$, we have only two cases: either $v=r$, or $v=k$ and $r=0$. In the former case, $e=q$ and in the later, $a_i=e+1=q$ for every $i$. The result follows. 
\end{proof}

\subsection{Case $k$ odd}

In this section, we show that if the number of vertices $k$ is odd, then there is a permutation of the sequence of 
Lemma \ref{lemma_max} such that its corresponding sub-polygon is asymmetric.

\begin{theorem}\label{th_odd} 
Given a regular polygon of $2n$ vertices inscribed in a circle and positive integers $k,q,r$ with $2n=kq+r$, $0\leq r < k$, the maximum area asymmetric sub-polygon with $k$ vertices is given by the following sequences:

For $r$ even, 
$$<\seq{q+1}{r/2},\;\seq{q}{(k-r-1)/2},\;\seq{q+1}{r/2},\;\seq{q}{(k-r+1)/2}>$$

For $r$ odd,
$$<\seq{q}{(k-r)/2},\;\seq{q+1}{(r-1)/2},\;\seq{q}{(k-r)/2},\;\seq{q+1}{(r+1)/2}>$$

\end{theorem}

\begin{proof}
We note first that both sequences are composed by $r$ occurrences of the element $q+1$ and $k-r$ occurrences of the element $q$. Therefore, both $k$-gons are maximum area by Lemma \ref{lemma_max}. Moreover, $a_i = a_{i+\frac{(k-1)}{2}}$ for $0\leq i<\frac{(k-1)}{2}$, we say that $a_{i+\frac{(k-1)}{2}}$ is the \emph{homologous element} of $a_i$ in the sequence and vice versa. Notice that, except for $a_{k-1}$, every element in the sequence has its homologous element.

We now prove that the encoded polygon is asymmetric. Let $A$ and $B$ be a partition into two subsequences of the proposed cyclic sequence so that $A$ contains $\frac{(k-1)}{2}$ consecutive elements and $B$ the other $\frac{(k+1)}{2}$ elements.

Let $S_A$ and $S_B$ be the sum of the elements of $A$ and $B$, respectively.
If $a_{k-1}\notin A$, then $B$ consists of the $\frac{(k-1)}{2}$ homologous elements of the corresponding of $A$ and $a_{k-1}$ and hence $S_B=S_A+a_{k-1}$. On the other hand, if $a_{k-1}\in A$, then $B$ contains $\frac{(k-3)}{2}$ homologous elements of the ones in $A$ plus two more elements, say $a_i,a_j$ (recall that $a_{k-1}$ has no homologous element). Therefore $S_B-S_A=a_i+a_j-a_{k-1}$ for some indices $i\not=j$. Regardless of the case, we have that $S_B>S_A$, $S_A<n$  and $S_B>n$. As a consequence, any sum of $\frac{(k-1)}{2}$ consecutive elements is different from $n$.

For an arbitrary partition of the sequence, we proceed in a similar way. If the sizes of $A$ and $B$ are arbitrary, we also have two cases. If $\vert A \vert \leq \frac{(k-1)}{2}$ then $S_A<n$ and if $\vert B \vert \leq \frac{(k-1)}{2}$ then $S_B<n$. In both cases, the sequences correspond to asymmetric polygons.

\end{proof}

\subsection{ Case $k$ even}

Unfortunately, if the number of vertices $k$ is even, the following result prevents us from seeking the solution as a permutation of the sequence of Lemma \ref{lemma_max}.

\begin{theorem}
If $k$ is even, then no permutation of the sequence 
$$<\seq{q+1}{r},\;\seq{q}{k-r}>$$
corresponds to an asymmetric polygon.
\end{theorem}
\begin{proof}
Set $f(i)=a_i+a_{i+1}+...+a_{i+k/2-1}$. We have $f(i+1)-f(i)=a_{i+k/2}-a_i$, which implies that $f(i+1)-f(i)\in \left\{-1,0,1\right\}$. If $f(0)=n$, then the sequence is non-asymmetric. Assume that $f(0)<n$. Thus, $f(k/2)>n$ ($f(0)+f(k/2)=2n)$. 
 Consider now the discrete function given by the integer values $f(0), f(1), ..., f(k/2)$. Since $\vert f(i+1)-f(i) \vert$ is equal to 1 or 0, $f(0)<n$ and $f(k/2)>n$, it follows that there exists an index $j$, $0<j<k/2$, such that $f(j)=n$ and the sequence generates a non-asymmetric polygon. The proof for the case $f(0)>n$ is analogous.
\end{proof}

In the following, we prove that it is possible to find the solution for the case $k$ even by considering the sequence that provides the second maximum area.

\begin{theorem}\label{th_even}
If $k$ is even, then any permutation of  the sequence
$$<\seq{q+1}{r+1},\;q-1,\;\seq{q}{k-r-2}>$$
corresponds to the sub-polygon of $k$-vertices with the second maximum area.
\end{theorem}
\begin{proof}

We first prove that if $a_i$ is an element of a sequence providing the $k$-gon with the second maximum area 
then $a_i\in \left\{ q-1,q,q+1,q+2 \right\}$ for $i=0,\ldots,k-1$.

Denote by $a_{max}$ and $a_{min}$ the values $\max\{a_0,a_1,..., a_{k-1}\}$ and $\min\{a_0,a_1,...,a_{k-1}\}$, respectively. Suppose that $a_{max}\geq q+3$. 
Note that if $a_{min}\geq q+1$, then $\sum^{k-1}_{i=0}a_{i}>2n$ which is a contradiction. Therefore, we assume that $a_{min} \leq q$.
In this case, switching $a_{max}$ by $a_{max}-1\geq q+2$ and $a_{min}$ by $a_{min}+1$ yields a larger polygon according to Lemma \ref{lemma_swap}. However, the new polygon cannot be the largest one because it contains an element $a_j\geq q+2$, in contradiction with Lemma \ref{lemma_max}.

Suppose that $a_{min}\leq q-2$. Following a similar argument to the above case, $a_{max}\geq q+1$ and by switching $a_{max}$ and $a_{min}$ by $a_{max}-1$ and $a_{min}+1 \leq q-1$, respectively, the area of the new polygon increases. It cannot be the largest one because it contains an element $a_j\leq q-1$.

Using similar arguments we can show that the elements $q+2$ and $q-1$ appear at most once in the $k$-gon 
with the second maximum area.

Let $A,B,C,D$ be the number of occurrences of $q-1,q,q+1,q+2$, respectively. The we have:
\begin{eqnarray*}
A+B+C+D & = & k\\
A(q-1)+B(q)+C(q+1)+D(q+2) & = & 2n\\
(A+B+C+D)q+(-A+C+2D) & = & kq+r\\
-A+C+2D & = & r\\
C & = & r+A-2D
\end{eqnarray*}

Now, the area of the $k$-gon we are looking for is:

$$S(A,B,C,D)=A\sin\left(\frac{\pi (q-1)}{n}\right)+ B\sin\left(\frac{\pi q}{n}\right)+
C\sin\left(\frac{\pi (q+1)}{n}\right)+
D\sin\left(\frac{\pi (q+2)}{n}\right)$$

Since $0\leq A,D \leq 1$, we have four cases to consider. Case $(A,D)=(0,0)$ corresponds with the $k$-gon of maximum area (Lemma \ref{lemma_max}). Hence, the areas of the remaining cases are:

\begin{displaymath}
\begin{array}{l}
S(1,B_1,C_1,1) = \sin\left(\frac{\pi (q-1)}{n}\right)+ (k-r-1)\sin\left(\frac{\pi q}{n}\right)+
(r-1)\sin\left(\frac{\pi (q+1)}{n}\right)+
\sin\left(\frac{\pi (q+2)}{n}\right)\\
S(0,B_2,C_2,1) = (k-r+1)\sin\left(\frac{\pi q}{n}\right)+
(r-2)\sin\left(\frac{\pi (q+1)}{n}\right)+
\sin\left(\frac{\pi (q+2)}{n}\right)\\
S(1,B_3,C_3,0) = \sin\left(\frac{\pi (q-1)}{n}\right)+ (k-r-2)\sin\left(\frac{\pi q}{n}\right)+
(r+1)\sin\left(\frac{\pi (q+1)}{n}\right)
\end{array}
\end{displaymath}

%

We claim that $S(1,B_1,C_1,1)<S(0,B_2,C_2,1)\leq S(1,B_3,C_3,0)$. To prove this claim, note that

\begin{displaymath}S(0,B_2,C_2,1)-S(1,B_1,C_1,1)=2\sin\left(\frac{\pi q}{n}\right)-\sin\left(\frac{\pi (q-1)}{n}\right)-\sin\left(\frac{\pi (q+1)}{n}\right)\ .
\end{displaymath}

By using Lemma \ref{lemma_swap} with $a=q+1, b=q-1$ we have that $S(0,B_2,C_2,1)-S(1,B_1,C_1,1)>0$. 
Notice that $2n\geq kq \geq 4q>2q=a+b$ and $a-b=2$ in this case.

To prove the other inequality, we use the following trigonometric identities:

$$\sin(\alpha)-\sin(\beta)=2\cos((\alpha + \beta)/2 )\sin((\alpha - \beta)/2 )$$
$$\sin(3\alpha)=3\sin(\alpha)-4\sin^{3}(\alpha)$$

Using these identities, we obtain the following:

\begin{align*}
S(1,B_3,C_3,0)-S(0,B_2,C_2,1)&=\sin\left(\frac{\pi (q-1)}{n}\right)+3\sin\left(\frac{\pi (q+1)}{n}\right) -3\sin\left(\frac{\pi q}{n}\right)-\sin\left(\frac{\pi (q+2)}{n}\right)\\
&=6\cos\left(\frac{\pi (2q+1)}{2n}\right)\sin\left(\frac{\pi}{2n}\right) -2\cos\left(\frac{\pi (2q+1)}{2n}\right)\sin\left(\frac{3\pi}{2n}\right)\\
&=2\cos\left(\frac{\pi (2q+1)}{2n}\right) \left( 3\sin\left(\frac{\pi}{2n}\right) -\sin\left(\frac{3\pi}{2n}\right) \right)\\
&=2\cos\left(\frac{\pi (2q+1)}{2n}\right)4\sin^{3}\left( \frac{\pi}{2n} \right)
\end{align*}


Because $2n=kq+r$, $k\geq 4$ and from the fact that $C_2=r-2\geq 0$, we get that $2n\geq 4q+2$ and $\frac{\pi (2q+1)}{2n}\leq \frac{\pi}{2}$. 
As a consequence, $S(1,B_3,C_3,0)-S(0,B_2,C_2,1)\geq 0$ yielding the result.
\end{proof}

The above result asserts that any sequence with coefficients $A=1$, $B=k-r-2$, $C=r+1$, $D=0$ provides the $k$-gon with the second maximum area. In the following result, we provide a sequence that in addition corresponds to 
an asymmetric $k$-gon. 

\begin{theorem}
The sequence corresponding to the maximum area asymmetric $k$-gon for $k$ even is as follows:
$$<q-1,\;\seq{q+1}{r/2},\;\seq{q}{(k-r-2)/2},\;\seq{q+1}{(r+2)/2},\;\seq{q}{(k-r-2)/2}>$$
 \end{theorem}

\begin{proof}
The proof is similar to that of Theorem \ref{th_odd}.
The sequence proposed, say $Q$, is a permutation of the sequence in Theorem \ref{th_even}. Therefore, its corresponding $k$-gon has the second maximum area. We show that it is also asymmetric. 

 Observe that for $0<i< k/2$, it holds that $a_i=a_{i+k/2}$. We say that $a_{i+k/2}$ and $a_i$ are homologous elements for $0<i< k/2$.
 
Let $A$ and $B$ be a partition into two subsequences of the cyclic sequence $Q$. 
First, suppose that $A$ (and $B$) consists of $k/2$ consecutive elements.
In this case, neither $A$ nor $B$ contains homologous elements. 
Let $S_A$ and $S_B$ be the sum of the elements of $A$ and $B$, respectively. If $a_0\in A$ then $a_{k/2}\in B$ and $S_A-S_B=a_0-a_{k/2}=-2\neq 0$. Similarly, if $a_0 \in B$ we have $S_A-S_B\neq 0$ proving our claim.

Suppose now that the subsequence $A$ contains less than $k/2$ consecutive elements. 
Then, every element in $A$ has its homologous element in $B$ except maybe for $a_0$ or $a_{k/2}$.
If $a_0 \in A$, then $S_B-S_A=a_{k/2}-a_0+a_u+a_v+ \ldots$, where $a_u,a_v, \ldots$ are elements of $B$ having no homologous element in $A$. In this case, $S_B>S_A$. The case $\vert B \vert<k/2$ is analogous.
Hence, there are no indices $i,j$, $i\neq j$, so that $a_i+a_{i+1}+...+a_j=n$. This completes the proof
of the theorem.
\end{proof}

\section{The general case}\label{sec:General Case}

D\'iaz-B\'a\~nez \etal~\cite{pilz} showed that the following problem is proved to be NP-hard: Given a set of $n$ line segments in the plane (or in a circle), find a maximum-area convex polygon having as vertices at least one endpoint from each segment. They also show that if the segments are pairwise disjoint, then the problem can be solved in polynomial time. A constrained version in which each segment joins two antipodal points on a circle was recently studied by Aicholzer \etal~\cite{aichholzer2013}. They showed that this constrained problem can be solved with a simple linear time algorithm and they asked about the hardness of the problem of selecting $k$ endpoints instead of $n$. 
As a consequence of the hardness result of D\'iaz-B\'a\~nez \etal~\cite{pilz}, this problem becomes NP-hard if we remove the antipodality constraint.
 
A dynamic programming algorithm presented in this section finds the asymmetric $k$-gon of maximum area in time $O(kn^4)$. For the special cases $k=3$ and $k=4$ we show how to solve the problem in linear time. 

\subsection{Dynamic programming algorithm}

Let $S:=\{p_0,p_0', p_1,p_1',\dots,p_{n-1},p_{n-1}'\}$ be the set of the endpoints of $n$ diameters in the unit circle and let $O$ be the center of this circle. We define a \emph{wedge} $(i,j)$ as a convex polygon being the convex hull of $O$ and a subsequence of points (need not be consecutive) of $S$ visited when walking clockwise along the circle from $p_i$ to $p_j$, see Figure~\ref{fig:def_wedge}. We say that $p_i$ and $p_j$ are anchors of a wedge $(i,j)$.

\begin{figure}[h]
\centering
\begin{subfigure}[b]{.26\textwidth}
\centering
\includegraphics[scale=.8, page=11]{double_wedge.pdf}
\caption{}
\label{fig:def_wedge}
\end{subfigure}%
\begin{subfigure}[b]{.37\textwidth}
\centering
\includegraphics[scale=.5, page=1]{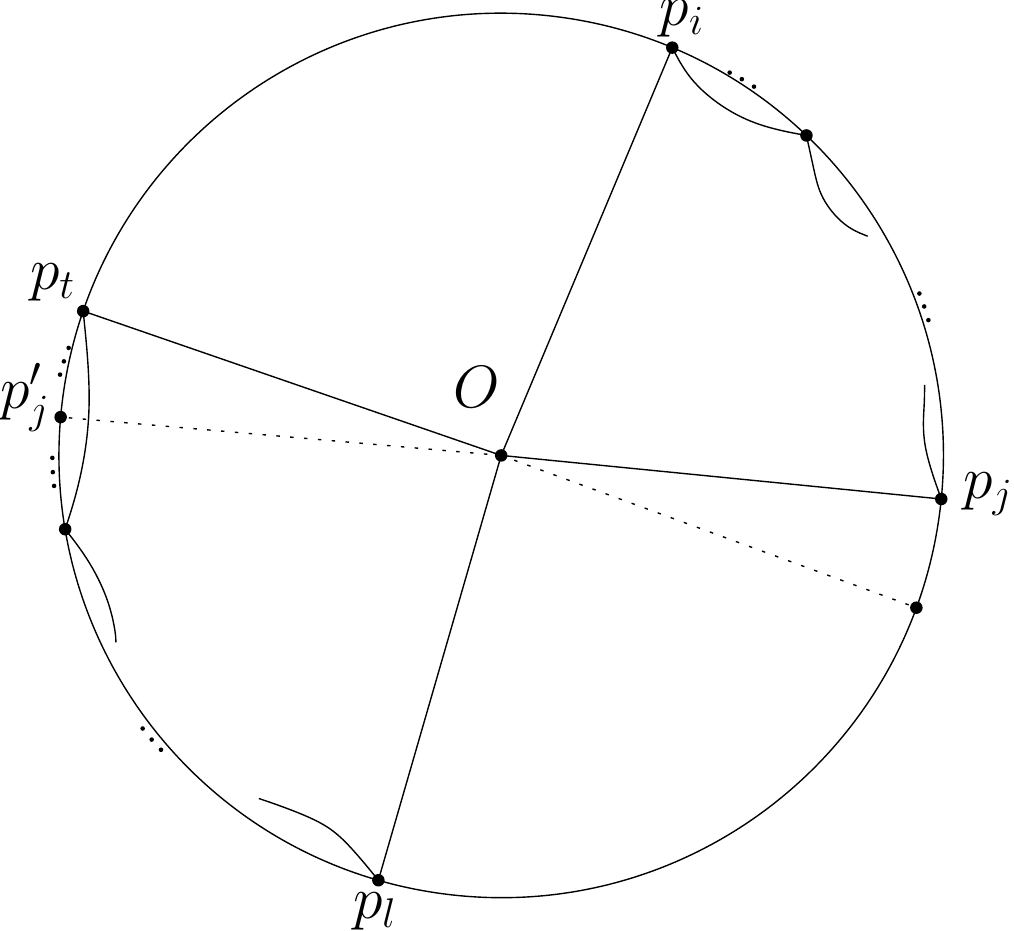}
\caption{}
\end{subfigure}%
\begin{subfigure}[b]{.37\textwidth}
\centering
\includegraphics[scale=.5, page=2]{double_wedge.pdf}
\caption{}
\end{subfigure}
\caption{The figure a) shows a wedge with anchors $p_i$ and $p_j$, b) and c) show the two possible configurations of a double-wedge $(i,j,l,t)$.}
\label{fig:double_wedge}
\end{figure}

A \emph{double-wedge} $(i,j,l,t)$ is the union of two disjoint wedges $(i,j)$ and $(l,t)$ such that: (1) this union contains no diameter and (2) either $p_t'$ lies clockwise between $p_i$ and $p_j$, or $p_j'$ lies clockwise between $p_l$ and $p_t$; see Figure~\ref{fig:double_wedge} for an illustration. The anchors of a double-wedge are the anchors of the wedges that form it.

\begin{lemma}\label{lemma:Double wedge and two triangles}
The maximum area asymmetric $k$-gon ($k\geq3$) on $S$ is formed by the union of a double-wedge $(i,j,l,t)$ and the triangles $Op_jp_l$ and $Op_ip_t$.
\end{lemma}
\begin{proof}
Let $Q$ be the maximum area asymmetric $k$-gon on $S$. 
Let $p_i$ be a vertex of $Q$ and let $p_i'$ be its antipodal point. 
Let $p_j$ (\emph{resp.} $p_l$) be the first vertex of $Q$ found when walking counterclockwise (\emph{resp.} clockwise) from $p_i'$ along the boundary of the circle. 
Let $p_t$ be the first vertex of $Q$ found when walking counterclockwise from $p_i$ along the boundary of the circle. See Figure \ref{fig:asym_dwedge}.

We prove by contradiction that $p_j$, $p_l$ and $p_t$ exist and are different from $p_i$. Let $S_Q\subseteq S$ be the set of vertices of $Q$ and its antipodal points and the asymmetric $k$-gon $Q$ of maximum area on $S$ is a maximum area antipodal on $S_Q$. Suppose that $p_j$ does not exist or matches with $p_i$, then all vertices of $Q$ lie in the same semicircle of the two defined by the diameter $p_ip_i'$, then $Q$ would not be a maximum area antipodal polygon on $S_Q$\cite{aichholzer2013}, this is a contradiction. The proof that $p_l$ and $p_t$ exist and are different from $p_i$ is analogous, note that $p_l$ can matches with $p_t$. 
The union of the wedges $(i,j)$ and $(l,t)$ contains no diameter because $Q$ is an asymmetric polygon. We now prove that $p_j'$ lies clockwise between $p_l$ and $p_t$ or $p_t'$ lies clockwise between $p_i$ and $p_j$. If $p_j'$ lies clockwise between $p_t$ and $p_i$ then $p_t'$ lies clockwise between $p_i$ and $p_j$, if $p_j'$ lies clockwise between $p_l$ and $p_t$ we are done, finally if $p_j'$ lies clockwise between $p_i'$ and $p_l$ we obtain a contradiction because all vertices of $Q$ would lie in the same semicircle of the two defined by the diameter $p_jp_j'$. Therefore, $Q$ is not a maximum area antipodal polygon on $S_Q$\cite{aichholzer2013}.
\end{proof}

\begin{figure}
\centering
\includegraphics[scale=.55, page=3]{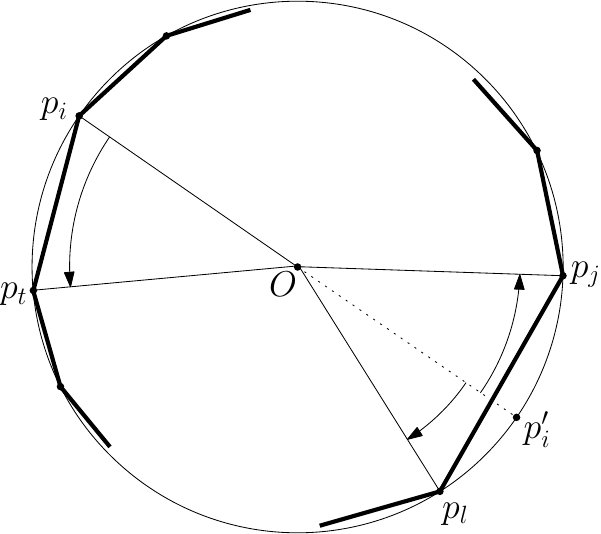}
\caption{A maximum area asymmetric $k$-gon on $S$ formed by the union of a double-wedge $(i,j,l,t)$ and the triangles $Op_jp_l$ and $Op_ip_t$.}
\label{fig:asym_dwedge}
\end{figure}

When a double-wedge uses $s$ points of $S$ we say it is an \emph{$s$-double-wedge}. For $3 \le s\le k$, let $f(i,j,l,t,s)$ be the area of the $s$-double-wedge of maximum area with anchors $(i,j,l,t)$. Notice that for $s=3$ either $p_i$ and $p_j$, or $p_l$ and $p_t$ coincide. 
By Lemma~\ref{lemma:Double wedge and two triangles}, the area of the maximum area asymmetric $k$-gon is equal to the maximum of $\{f(i,j,l,t,k)+A_{Op_jp_l}+A_{Op_ip_t}\}$, where $(i,j,l,t)$ goes over all the possible anchors of a $k$-double-wedge. Therefore, having computed $f(i,j,l,t,k)$ for all possible $k$-double-wedges, the maximum area asymmetric $k$-gon can be found in time $O(n^4)$.

\begin{lemma}
All possible values of $f(i,j,l,t,s)$ with $3\leq s\leq k$ can be computed in time $O(kn^5)$.
\end{lemma}
\begin{proof}
For $s=3$, if $i=j$ then $f(i,j,l,t,3)=A_{Op_lp_t}$ else $f(i,j,l,t,3)=A_{Op_ip_j}$. All possible values of $f(i,j,l,t,3)$ can be computed in time $O(n^3)$. For fixed values $i$, $j$, $l$, $t$ and $s$ ($s>3$), the value of $f(i,j,l,t,s)$ can be computed in time $O(n)$. If $p_t'$ lies clockwise between $p_i$ and $p_j$ then: $$f(i,j,l,t,s)=\max\{f(i,m,l,t,s-1) + A_{Op_mp_j}\}$$ 
for each $p_m$ between $p_i$ and $p_j$ in clockwise, $p_m\neq p_l'$ and $p_m\neq p_t'$. Otherwise, if $p_j'$ lies clockwise between $p_l$ and $p_t$ then: $$f(i,j,l,t,s)=\max\{f(i,j,l,m,s-1) + A_{Op_mp_t}\}$$ for each $p_m$ between $p_l$ and $p_t$ in clockwise, $p_m\neq p_i'$ and $p_m\neq p_j'$. In fact, all possible values of $f(i,j,l,t,s)$ with a fixed value $s$ ($s>3$) can be computed in time $O(n^5)$ having computed the values $f(i,j,l,t,s-1)$. Then, by iterating $s$ from $3$ to $k$, all possible values of $f(i,j,l,t,s)$ can be computed in time $O(kn^5)$.
\end{proof}

We now show how to improve this running time by computing all possible values of $f(i,j,l,t,s)$ for a fixed value $s>3$ in time $O(n^4)$. The following lemma is depicted in Figure~\ref{fig:ordered_i_r_q}.

\begin{figure}[h]
\centering
\begin{subfigure}{.5\textwidth}
\centering
\includegraphics[scale=.6, page=7]{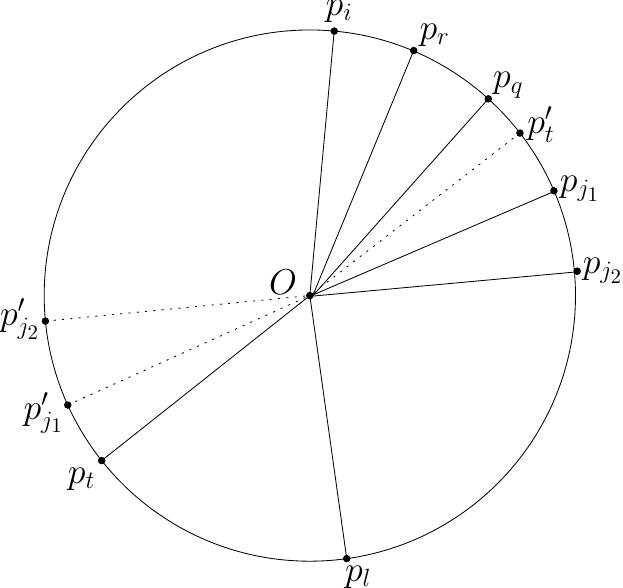}
\caption{}
\label{fig:ordered_i_r_q}
\end{subfigure}%
\begin{subfigure}{.5\textwidth}
\centering
\includegraphics[scale=.6, page=10]{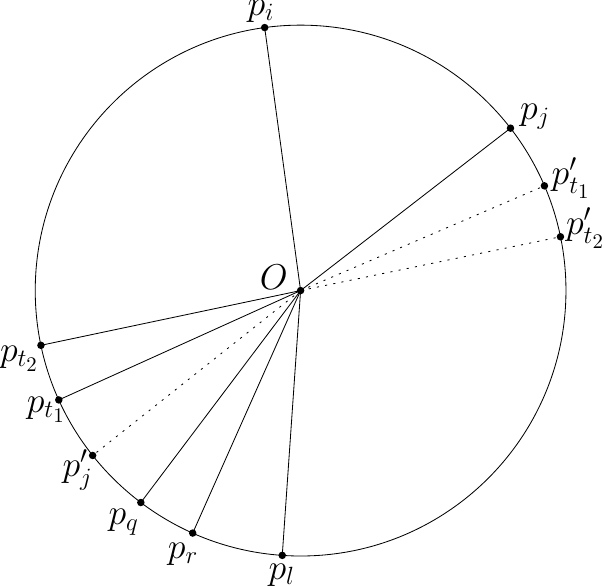}
\caption{}
\label{fig:ordered_l_r_q}
\end{subfigure}
\caption{The figures a) and b) show the statements of Lemmas \ref{thm:ordered_i_r_q} and \ref{thm:ordered_l_r_q}, respectively.}

\end{figure}

\begin{lemma}
\label{thm:ordered_i_r_q}
Let $f(i,j_1,l,t,s)$ and $f(i,j_2,l,t,s)$ be the areas of the maximum area $s$-double-wedges with anchors $(i,j_1,l,t)$ and $(i,j_2,l,t)$, respectively, such that $s>3$ and $p_t'$ lies clockwise between $p_i$ and $p_{j_1}$. 
Let $p_i,\dots,p_r,p_{j_1}$ and $p_i,\dots,p_q,p_{j_2}$ be the subsequences in $S$ of the wedges $(i,j_1)$ and $(i,j_2)$, respectively. 
If $p_{j_1}$ precedes $p_{j_2}$ when walking clockwise along the circle from $p_i$, then $p_r$ also precedes $p_q$.
\end{lemma}
\begin{proof}
If $p_q$ lies clockwise between $p_{j_1}$ (including it) and $p_{j_2}$ we are done. Now, consider the case when $p_q$ lies clockwise between $p_i$ and $p_{j_1}$.
The subsequences $p_i,\dots,p_r,p_{j_1}$ and $p_i,\dots,p_q,p_{j_2}$ are part of the maximum area $s$-double-wedges $(i,j_1,l,t)$ and $(i,j_2,l,t)$, respectively. Therefore we have
\begin{eqnarray}
f(i,j_1,l,t,s) & =f(i,r,l,t,s-1)+A_{Op_rp_{j_1}}& \nonumber\\ 
&\geq f(i,m,l,t,s-1)+A_{Op_mp_{j_1}}&\text{ for each }p_m\text{ lies clockwise between }p_i\text{ and }p_{j_1}\label{eq:eq1}
\end{eqnarray}%
\begin{eqnarray}
f(i,j_2,l,t,s) & =f(i,q,l,t,s-1)+A_{Op_qp_{j_2}}&\nonumber\\
&\geq f(i,m,l,t,s-1)+A_{Op_mp_{j_2}}&\text{ for each }p_m\text{ lies clockwise between }p_i\text{ and }p_{j_2}
\label{eq:eq2}
\end{eqnarray}
From \ref{eq:eq1} and \ref{eq:eq2} we obtain the following inequalities:
$$f(i,r,l,t,s-1)+A_{Op_rp_{j_1}}\geq f(i,q,l,t,s-1)+A_{Op_qp_{j_1}}$$
$$f(i,r,l,t,s-1)+A_{Op_rp_{j_2}}\leq f(i,q,l,t,s-1)+A_{Op_qp_{j_2}}$$
adding both we obtain that
\begin{equation}
A_{Op_rp_{j_2}}-A_{Op_rp_{j_1}}\leq A_{Op_qp_{j_2}}-A_{Op_qp_{j_1}}
\label{eq:area_triang}
\end{equation}

Suppose that $p_q$ and $p_r$ are inverted. 
Let $z$, $x$ and $y$ be the angles between $p_q$, $p_r$, $p_{j_1}$ and $p_{j_2}$ (see Figure \ref{fig:implied_angles}), then Inequality~\ref{eq:area_triang} can be written as:
$$\sin(x+y)-\sin x \leq \sin(x+y+z)-\sin(z+x)$$
$$2\cos(x+\dfrac{y}{2})\sin(\dfrac{y}{2})\leq 2\cos(x+z+\dfrac{y}{2})\sin(\dfrac{y}{2})$$
\begin{equation}
\cos(x+\dfrac{y}{2})\leq \cos(x+z+\dfrac{y}{2})
\label{eq:area_contradiction}
\end{equation}
Because $z+x+y<\pi$ and $\cos(\cdot)$ is decreasing in the interval $(0, \pi)$, Inequality~\ref{eq:area_contradiction} yields a contradiction.

\begin{figure}[h!]
\centering
\includegraphics[scale=.6, page=9]{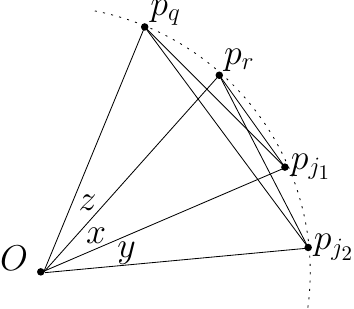}
\caption{This figure shows the triangles involved in the proof of Theorem~\ref{thm:ordered_i_r_q}.}
\label{fig:implied_angles}
\end{figure}
\end{proof}

The following lemma has an analogous proof to that of Lemma~\ref{thm:ordered_i_r_q} and is depicted in Figure~\ref{fig:ordered_l_r_q}.

\begin{lemma}
\label{thm:ordered_l_r_q}
Let $f(i,j,l,t_1,s)$ and $f(i,j,l,t_2,s)$ be the areas of the maximum area $s$-double-wedges with anchors $(i,j,l,t_1)$ and $(i,j,l,t_2)$, respectively, such that $s>3$ and $p_j'$ lies clockwise between $p_l$ and $p_{t_1}$. 
Let $p_l,\dots,p_r,p_{t_1}$ and $p_l,\dots,p_q,p_{t_2}$ be the subsequences in $S$ of the wedges $(l,t_1)$ and $(l,t_2)$, respectively. 
If $p_{t_1}$ precedes $p_{t_2}$ when walking clockwise along the circle from $p_l$, then $p_r$ also precedes $p_q$.
\end{lemma}

\begin{theorem}
Let $S:=\{p_0,p_0', p_1,p_1',\dots,p_{n-1},p_{n-1}'\}$ be the set of the endpoints of $n$ diameters in the unit circle and let $k$ be an integer such that $k<n$. The asymmetric $k$-gon of maximum area can be found in time $O(kn^4)$.
\end{theorem}
\begin{proof}
Having computed $f(i,j,l,t,k)$ for all possible $k$-double-wedge, the maximum area asymmetric $k$-gon can be found in time $O(n^4)$. We show how to compute $f(i,j,l,t,s)$ for all $s$-double-wedge with $3\leq s \leq k$ in time $O(kn^4)$.

For $s=3$, $f(i,j,l,t,3)$ can be computed in time $O(n^3)$. For $s>3$, if $p_t'$ lies clockwise between $p_i$ and $p_j$, $f(i,j,l,t,s)=\max\{f(i,m,l,t,s-1) + A_{Op_mp_j}\}$, for each $p_m$ between $p_i$ and $p_j$ in clockwise, $p_m\neq p_l'$ and $p_m\neq p_t'$. Suppose that we have computed for $i$, $j$, $l$ and $t$ the value $r$ such that $f(i,j,l,t,s)=f(i,r,l,t,s-1) + A_{Op_rp_j}$. Then, to compute $f(i,j+1,l,t,s)$ we do not need test each $p_m$ between $p_i$ and $p_{j+1}$, we need only to test each $p_m$ in the interval from $p_r$ to $p_{j+1}$. In fact, we can compute $f(i,j,l,t,s)$ for all $p_j$ that follows $p_t'$ in clockwise for fixed values $i$, $l$, $t$ and $s$ in time $O(n)$. The other case, when $p_j'$ lies clockwise between $p_l$ and $p_t$ is analogous. In fact, all possible values of $f(i,j,l,t,s)$ with a fixed value $s$ ($s>3$) can be computed in time $O(n^4)$ having computed the values $f(i,j,l,t,s-1)$. Then, iterating $s$ from $3$ to $k$, all possible values of $f(i,j,l,t,s)$ can be computed in time $O(kn^4)$.
\end{proof}

\subsection{Special case $k=3$}
Let $S:=\{p_0, p_0', p_1, p_1', \dots, p_{n-1}, p_{n-1}'\}$ be the set of endpoints of $n$ diameters in the unit circle, $n\geq 3$.
In the following, we prove that the triangle of maximum area on $S$ contains no diameter. Using this characterization, we show how to find the largest triangle in linear time. We use $\arc{p_ip_j}$ to denote the clockwise arc from $p_i$ to $p_j$ and $\angle{p_ip_jp_k}$ to denote the inner angle of the triangle $p_ip_jp_k$ with apex $p_j$.

\begin{lemma}
A maximal triangle with vertices in $S$ contains no diameter.
\end{lemma}
\begin{proof}
Suppose for a contradiction that the maximal triangle contains a diameter. Let $p_ip_jp_i'$ be the maximal triangle. Let $p_l$ and $p_l'$ be another diameter of points in $S$. If $p_l$ is in $\arc{p_i p_j}$ then $p_l'$ is in $\arc{p_i' p_j'}$. The triangles $p_ip_jp_i'$ and $p_ip_jp_l'$ have a common base, the segment $p_i p_j'$, the height of $p_ip_jp_i'$ from $p_i'$ is smaller than the height of $p_ip_jp_l'$ from $p_l'$; see Figure~\ref{fig:no_diameter}. Therefore, $A_{p_ip_jp_l'}>A_{p_ip_jp_i'}$ which is a contradiction. The proof in the case $p_l$ is in $\arc{p_j p_i'}$ is analogous.
\end{proof}

\begin{figure}[h!]
\centering
\includegraphics[scale=0.525]{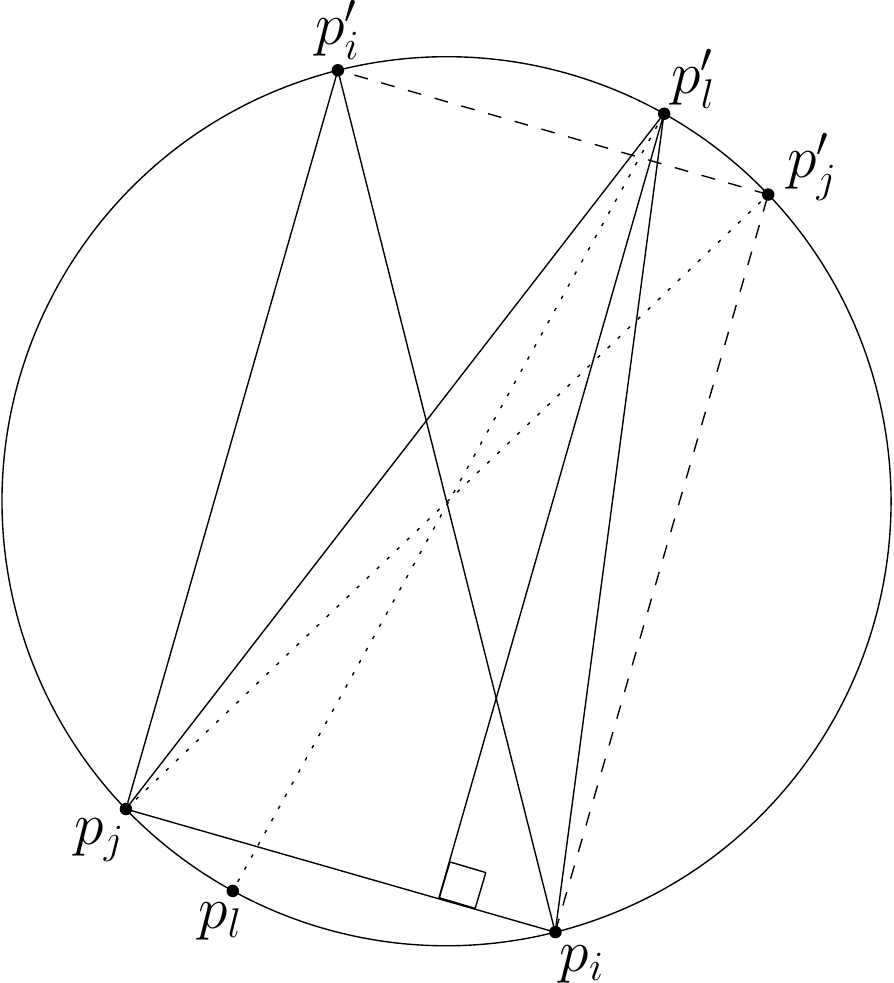}
\caption{$A_{p_l'p_jp_i}>A_{p_ip_i'p_j}$ because the height of $p_l'p_jp_i$ from $p_l'$ is greater than the height of $p_ip_i'p_j$ from $p_i'$.}
\label{fig:no_diameter}
\end{figure}

Note that each triangle contains either two angles greater or equal than $2\pi/3$, or two angles smaller or equal than $2\pi/3$; we call such pair of angles the \emph{base angles} of the triangle. A \emph{critical vertex} is a vertex supporting a base angle. The following lemma is key to our algorithm.

\begin{lemma}
\label{thm:anchored}
Let $c_1$, $c_2$, $p$, $b_2$ and $b_1$ be points in the unit circle that lie in this order when walking counterclockwise along the circle from $c_1$.
\begin{enumerate}
\item If $\angle pb_1c_1\leq2\pi/3$ and $\angle pc_1b_1\leq2\pi/3$, then the maximal triangle that contains $p$ is $c_1pb_1$
\item If $\angle pb_2c_2\geq 2\pi/3$ and $\angle pc_2b_2\geq 2\pi/3$, then the maximal triangle that contains $p$ is $c_2pb_2$
\end{enumerate}

\end{lemma}
\begin{proof}
Using similar triangles we can show the following inequalities for the first case (see Figure~\ref{fig:anchored_over}): $$A_{pc_2b_2}<A_{pc_2b_1}<A_{pc_1b_1}$$
$$A_{pc_2b_2}<A_{pc_1b_2}<A_{pc_1b_1}\ ,$$
and these ones for the other case (see Figure~\ref{fig:anchored_under}):
$$A_{ pc_2b_2}>A_{ pc_2b_1}>A_{ pc_1b_1}$$
$$A_{ pc_2b_2}>A_{ pc_1b_2}>A_{ pc_1b_1}\ .$$
\end{proof}

\begin{figure}[h]
\centering
\begin{subfigure}{.5\textwidth}
\centering
\includegraphics[scale=.7]{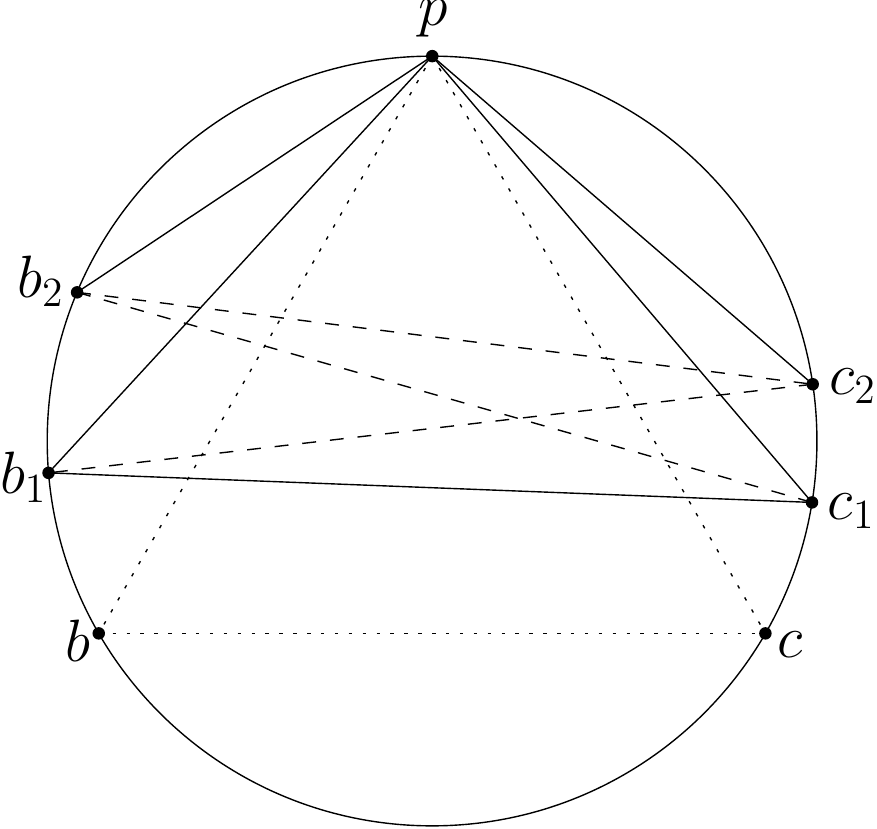}
\caption{}
\label{fig:anchored_over}
\end{subfigure}%
\begin{subfigure}{.5\textwidth}
\centering
\includegraphics[scale=.7]{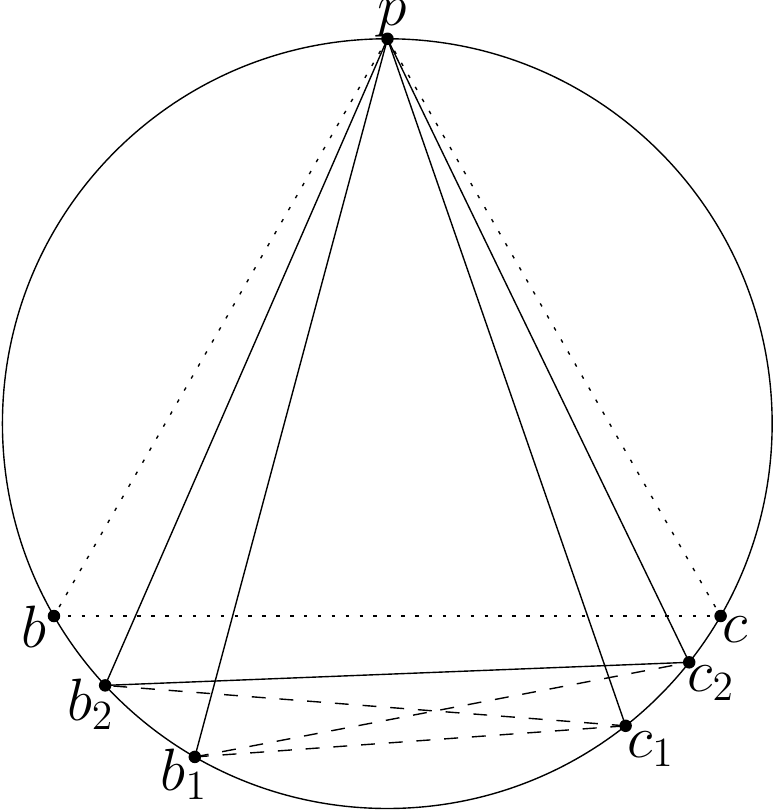}
\caption{}
\label{fig:anchored_under}
\end{subfigure}%
\caption{This figure shows triangles with $p$ as critical vertex. In both cases the maximal triangle is the closest to the equilateral triangle $pbc$.}
\end{figure}

\begin{theorem}
Let $S:=\{p_0, p_0', p_1, p_1', \dots, p_{n-1}, p_{n-1}'\}$ be the set of endpoints of $n$ diameters in the unit circle, $n\geq 3$. The triangle of maximum area with vertices in $S$ can be found in linear time.
\end{theorem}
\begin{proof}The idea of our algorithms is to iterate among the endpoints of the $n$ diameters in clockwise order computing for each $p_i$ the maximal triangle with $p_i$ as critical vertex. Let $M$ be the triangle of maximum area with vertices in $S$. Then $M=\max\{T_i\}$, where $T_i$ is the maximal triangle with $p_i$ as critical vertex. Lemma~\ref{thm:anchored} allows us to find $T_i$ comparing just two triangles $p_ip_{j_i}p_{m_i}$ and $p_ip_{k_i}p_{l_i}$ (see Figure~\ref{fig:max_cands}), where $j_i\leq k_i < l_i\leq m_i$ and the following inequalities hold: 
\begin{eqnarray}
\arc{p_ip_{j_i}}\leq 2\pi/3<\arc{p_ip_{{j_i}+1}}
\label{eq_ji}\\
\arc{p_ip_{k_i}}\geq 2\pi/3>\arc{p_ip_{{k_i}-1}}
\label{eq_ki}\\
\arc{p_ip_{l_i}}\leq 4\pi/3<\arc{p_ip_{{l_i}+1}}
\label{eq_li}\\
\arc{p_ip_{m_i}}\geq 4\pi/3>\arc{p_ip_{{m_i}-1}}
\label{eq_mi}
\end{eqnarray}
Note that $j_{i+1}\geq j_i$, $k_{i+1}\geq k_i$, $l_{i+1}\geq l_i$ and $m_{i+1}\geq m_i$.
If we have for a point $p_i$ the points $p_{j_i}$, $p_{k_i}$, $p_{l_i}$ and $p_{m_i}$, then moving from these points in clockwise we can find for $p_{i+1}$ the points $p_{j_{i+1}}$, $p_{k_{i+1}}$, $p_{l_{i+1}}$ and $p_{m_{i+1}}$. Therefore, we can compute the maximal triangle with $p_i$ as critical vertex for all $p_i$ in $S$ in linear time.
\end{proof}

\begin{figure}
\centering
\includegraphics[scale=.7]{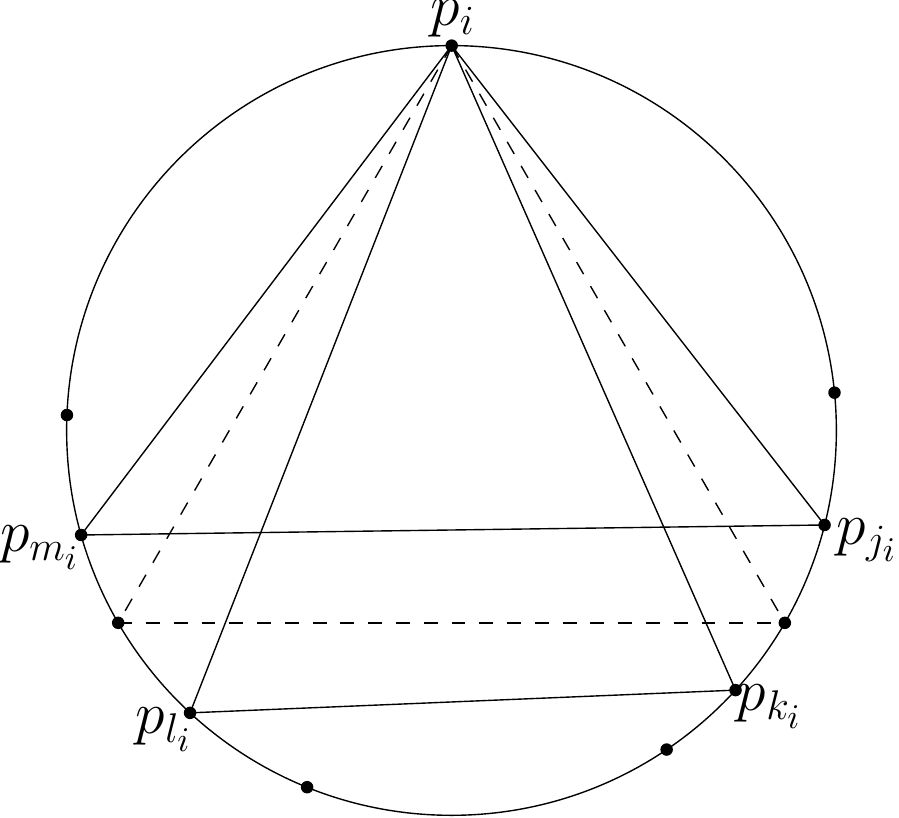}
\caption{The two candidates to be the maximal triangles with critical vertex $p_i$, $p_ip_{j_i}p_{m_i}$ and $p_ip_{k_i}p_{l_i}$.}
\label{fig:max_cands}
\end{figure}
 
\subsection{Special case $k=4$}
Unlike the triangle, the quadrilateral of maximum area could contain diameters. However, we prove that the asymmetric quadrilateral of maximum area satisfies a property that allows us to find it in linear time.

\begin{lemma}
\label{thm:k=4}
Let $S:=\{p_0, p_0', p_1, p_1', \dots, p_{n-1}, p_{n-1}'\}$ be the set of endpoints of $n$ diameters in the unit circle, $n\geq 4$. Let $p_i$, $p_l$, $p_j'$ and $p_t'$ be the vertices of a maximal quadrilateral. Then the arcs $\arc{p_ip_j}$, $\arc{p_lp_t}$, $\arc{p_i'p_j'}$ and $\arc{p_l'p_t'}$ contain no other points of $S$, see Figure~\ref{fig:k=4}.
\end{lemma}

\begin{proof}
Certainly $p_l$ is the farthest point of $S$ from ${p_ip_j'}$ in $\arc{p_ip_j'}$ or $p_t'$ is the farthest point of $S$ from ${p_ip_j'}$ in $\arc{p_j'p_i}$. Suppose that $p_l$ ($p_t'$) is the farthest point of $S$ from ${p_ip_j'}$ in $\arc{p_ip_j'}$ ($\arc{p_j'p_i}$), then $p_l'$ ($p_t$) is the farthest point of $S$ from ${p_ip_j'}$ in $\arc{p_j'p_i}$ ($\arc{p_ip_j'}$). In fact $p_t'$ ($p_l$) is the second farthest point of $S$ from ${p_ip_j'}$ in $\arc{p_j'p_i}$ ($\arc{p_ip_j'}$). Hence the points $p_l$ and $p_t$ ($p_l'$ and $p_t'$) are consecutive in clockwise. Analogously we can prove that $p_i$ and $p_j$ ($p_i'$ and $p_j'$) are consecutive .
\end{proof}

\begin{figure}[h]
\centering
\includegraphics[scale=1]{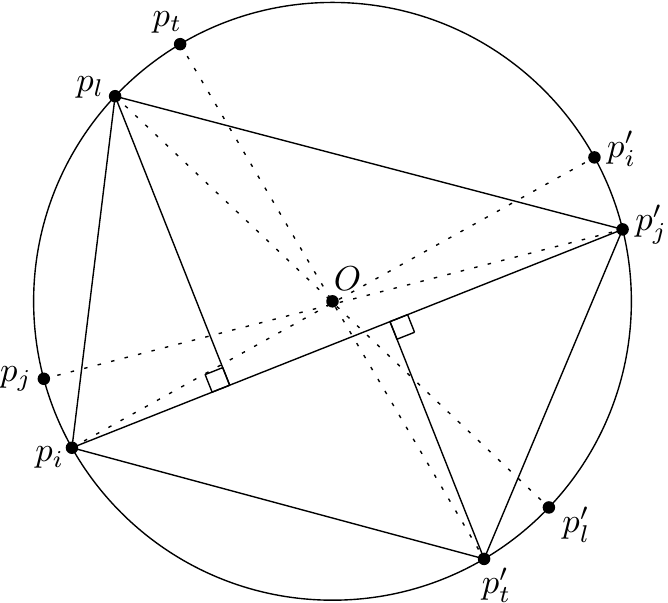}
\caption{Illustration of lemma~\ref{thm:k=4}}
\label{fig:k=4}
\end{figure}

Note that the four possible asymmetric quadrilateral with vertices in the subset $\{p_i$, $p_j$, $p_l$, $p_t$, $p_i'$, $p_j'$, $p_l'$, $p_t'\}$ have the same area because the segments ${p_ip_j'}$ and ${p_i'p_j}$ (${p_lp_t'}$ and ${p_l'p_t}$) are parallel and equal. As a consequence of Lemma~\ref{thm:k=4}, $p_j=p_{i+1}$ and $p_t=p_{l+1}$. From the above lemma we infer the following corollary.

\begin{corollary}
\label{cor:quad_diagonal}
Let $S:=\{p_0, p_0', p_1, p_1', \dots, p_{n-1}, p_{n-1}'\}$ be the set of endpoints of $n$ diameters in the unit circle, $n\geq 4$. Let $Q$ be a maximal asymmetric quadrilateral on $S$. If $p_i\in S$ is a vertex of $Q$ then its opposite vertex in $Q$ is $p_{i+1}'$ or $p_{i-1}'$.
\end{corollary}

\begin{lemma}
\label{thm:lastk=4}
Let $a$, $b$, $c$, $d$, $a'$, $e$, $f$ and $d'$ be clockwise points in a circle, where ${aa'}$ and ${dd'}$ are diameters. If $d$ is farther than $c$ from ${ae}$, then $d$ is farther than $c$ from ${bf}$.
\end{lemma}
\begin{proof}
Let $m$ be the intersection point of ${bd}$ and ${cf}$, let $n$ be the intersection point of ${ad}$ and ${ce}$. Because $d$ is farther than $c$ from ${ae}$, $A_{ace}<A_{ade}$. Because triangle $ane$ is common for the triangles $ace$ and $ade$, $A_{anc}<A_{dne}$. 
Since the triangles $anc$ and $dne$ are similar, ${ac}<{de}$. Moreover, ${bc}<{ac}$ and ${de}<{df}$ as they are in the same semicircle, respectively. From these inequalities, we infer that ${bc}<{df}$.  
Because the triangles $bmc$ and $dmf$ are similar, we know that $A_{bmc}<A_{dmf}$. Since triangle $bmf$ is common for the triangles $bcf$ and $bdf$, $A_{bcf}<A_{bdf}$. Therefore, $d$ is farther than $c$ from ${bf}$; see Figure~\ref{fig:c_d}.
\end{proof}

\begin{figure}[h]
\centering
\includegraphics[scale=1, page=4]{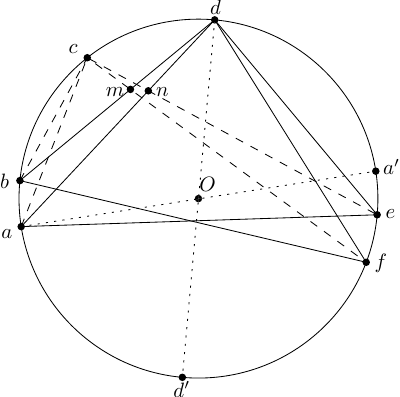}
\caption{Representation of lemma~\ref{thm:lastk=4}}
\label{fig:c_d}
\end{figure}

We are now ready to show the main result of this section.

\begin{theorem}
Let $S:=\{p_0, p_0', p_1, p_1', \dots, p_{n-1}, p_{n-1'}\}$ be the set of endpoints of $n$ diameters in the unit circle, $n\geq 4$. The asymmetric quadrilateral of maximum area can be found in linear time.
\end{theorem}
\begin{proof}
Let $Q$ be the asymmetric quadrilateral of maximum area on $S$ and let $Q_i=p_ip_lp_{i+1}'p_{l+1}'$ be a maximal asymmetric quadrilateral with ${p_ip_{i+1}'}$ as a diagonal, using Corollary~\ref{cor:quad_diagonal} we have that $Q=\max\{Q_i\}$. 
The idea of our algorithm is to iterate among the endpoints of the $n$ diameters in clockwise order and compute, for each $p_i$, the maximal asymmetric quadrilateral $Q_i$ with $p_ip_{i+1}'$ as a diagonal. Suppose that for an arbitrary point $p_i\in S$ we have computed $Q_i=p_ip_lp_{i+1}'p_{l+1}'$ where $p_l$ and $p_{l+1}'$ are the farthest points (the farthest and the second farthest maintaining the asymmetry) from $p_ip_{i+1}'$ on $S$. To compute $Q_{i+1}$ we need to find $p_m$ such that $p_m$ and $p_{m+1}'$ are the farthest points from $p_{i+1}p_{i+2}'$ on $S$, using Lemma~\ref{thm:lastk=4} we can infer that $p_m$ and $p_{m+1}'$ lie on $\arc{p_lp_{i+2}'}$ and $\arc{p_{l+1}'p_{i+1}}$, respectively (see Figure~\ref{fig:i_l_m}). Therefore, making a clockwise iteration on $S$ we can find all the maximal asymmetric quadrilaterals in linear time, computing for each $p_i$ the farthest points (the farthest and the second farthest maintaining the asymmetry) $p_l$ and $p_{l+1}$ from $p_ip_{i+1}'$ on $S$ using the farthest points found in the previous step of the iteration.
\end{proof}

\begin{figure}[h]
\centering
\includegraphics[scale=1, page=3]{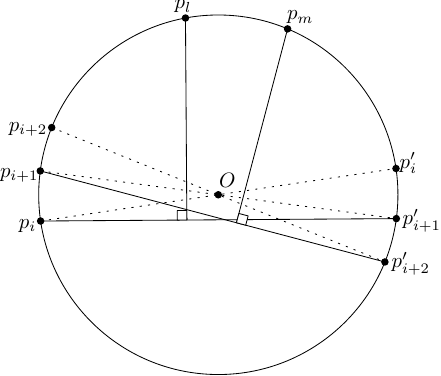}
\caption{$p_l$ is the farthest point of $S$ from ${p_ip_{i+1}'}$ and $p_m$ is the farthest point of $S$ from ${p_{i+1}p_{i+2}'}$.}
\label{fig:i_l_m}
\end{figure}

\section{Acknowledgments}

The problems studied here were introduced and partially solved during a visit
to University of La Havana, Cuba in November 2013. We thank the project COFLA: Computational analysis of the Flamenco music (FEDER P09-TIC-4840) for posing us the basic problem studied in this paper.

\bibliographystyle{plain}
\bibliography{antipodalbib}

\end{document}